\documentclass[11pt,a4paper]{amsart}

\usepackage[utf8]{inputenc}
\usepackage[T1]{fontenc}
\usepackage{amsmath,amssymb,amsthm}
\usepackage{hyperref}
\usepackage{xcolor}

\newtheorem{innercustomthm}{Theorem}
\newenvironment{customthm}[1]
  {\renewcommand\theinnercustomthm{#1}\innercustomthm}
  {\endinnercustomthm}
  
\newtheorem{innercustomcor}{Corollary}
\newenvironment{customcor}[1]
  {\renewcommand\theinnercustomcor{#1}\innercustomcor}
  {\endinnercustomcor}

\newtheorem{theorem}{Theorem}[section]
\newtheorem*{theorem*}{Theorem}
\newtheorem{lemma}[theorem]{Lemma}
\newtheorem{corollary}[theorem]{Corollary}
\newtheorem{proposition}[theorem]{Proposition}

\newtheorem{problem}[theorem]{Problem}

\newtheorem{conjecture}[theorem]{Conjecture}

\theoremstyle{remark}

\numberwithin{equation}{section}

\DeclareMathOperator{\divg}{div}
\DeclareMathOperator{\Hess}{Hess}
\DeclareMathOperator{\tr}{tr}

\newcommand{\R}{\mathbb{R}}
\newcommand{\gl}{\nabla^{L}}
\newcommand{\lapL}{\Delta_{L}}

\title[CMC foliations on Euclidean spaces are minimal foliations]{CMC foliations on Euclidean spaces are minimal foliations}

\author[J. E. Sampaio]{Jos\'e Edson Sampaio}
\author[E. C. da Silva]{Euripedes Carvalho da Silva}

\address{Jos\'e Edson Sampaio:  
              Departamento de Matem\'atica, Universidade Federal do Cear\'a,
	      Rua Campus do Pici, s/n, Bloco 914, Pici, 60440-900, 
	      Fortaleza-CE, Brazil. \newline   
              E-mail: {\tt edsonsampaio@mat.ufc.br}                    
}
\address{Euripedes Carvalho da Silva: Departamento de Matem\'atica, Instituto Federal de Educa\c{c}\~ao, Ci\^encia e Tecnologia do Cear\'a,
 	      Av. Parque Central, 1315, Distrito Industrial I, 61939-140, 
 	      Maracana\'u-CE, Brazil. \newline 
 	      and 	      Departamento de Matem\'atica, Universidade Federal do Cear\'a,
	      Rua Campus do Pici, s/n, Bloco 914, Pici, 60440-900, 
	      Fortaleza-CE, Brazil.
               E-mail: {\tt euripedes.carvalho@ifce.edu.br}
 } 

\thanks{The first named author was partially supported by CNPq-Brazil grants 303375/2025-6, 404611/2025-7, and 405618/2023-9. The second named author was partially supported by CNPq-Brazil grants 400043/2025-4 and 405618/2023-9. This work was supported by the Serrapilheira Institute (grant number Serra -- R-2110-39576).
}
\begin{document}

\begin{abstract}

In this article, we provide a complete answer to one of the most classical conjectures in differential geometry (of foliations). More precisely, we present a complete positive answer to the conjecture, due to Meeks III, P\'erez, and Ros (2008), which states that every foliation on $\mathbb{R}^{n+1}$ by (possibly varying) CMC hypersurfaces is a foliation by minimal hypersurfaces. Moreover, if $n\leq 6$, such a CMC foliation must consist of parallel hyperplanes. 
\end{abstract}

\maketitle
\tableofcontents
\section{Introduction}

In 1986, Solomon \cite{Solomon-1986} posed the following problem:
\begin{problem}\label{problem:minimal_solomon}
Is every foliation of $\mathbb{R}^{n+1}$, $n\leq 7$, by properly embedded minimal hypersurfaces a foliation by parallel affine hyperplanes? 
\end{problem}
In the case that the leaves are properly embedded, Problem \ref{problem:minimal_solomon} has an affirmative answer if $n+1<8$, and it has a negative answer if $n+1>8$, since, in those dimensions, there are minimal graphs that are not hyperplanes. In the critical dimension $n+1= 8$, Solomon in \cite{Solomon-1986} stated that it appears quite difficult to settle this question, even in the case where the leaves are properly embedded.
When the leaves are not supposed to be properly embedded, the problem is much more complicated.

This article deals with the following long-standing and more complicated conjectures due to Meeks III, P\'erez, and Ros (see \cite[Conjecture 5.1]{MeeksPR:2008}):
\begin{conjecture}\label{conjecture:cmc_solomon}
Every foliation of $\mathbb{R}^{n+1}$, $n\leq 7$, by (possibly varying) CMC hypersurfaces is a foliation by parallel affine hyperplanes.
\end{conjecture}

\begin{conjecture}\label{conjecture:yau-docarmo-solomon}
Every foliation of $\mathbb{R}^{n+1}$ by (possibly varying) CMC hypersurfaces is a foliation by minimal hypersurfaces.
\end{conjecture}
Although these conjectures were only explicitly stated in 2008, many mathematicians have been working on them since at least the 1980s. For instance, in 1987, Barbosa, Gomes, and Silveira in \cite[Theorem 3.12]{BarbosaGS-1987} proved that Conjectures \ref{conjecture:cmc_solomon} and \ref{conjecture:yau-docarmo-solomon} have positive answers when $n=2$ and under the extra hypothesis that all the leaves of the foliation have the same constant mean curvature. In 1988, Meeks III in \cite[Theorem 4.1]{Meeks-1988} presented a beautiful positive answer for Conjectures \ref{conjecture:cmc_solomon} and \ref{conjecture:yau-docarmo-solomon} when $n=2$. 
Still in \cite{MeeksPR:2008}, Meeks III, P\'erez and Ros proved that Conjecture \ref{conjecture:yau-docarmo-solomon} has a positive answer if $n\leq 4$.
So far, to the authors' knowledge, the aforementioned results from \cite{Meeks-1988} and \cite{MeeksPR:2008} are unique cases where these conjectures are known to be true without any additional hypotheses.
One can find other partial answers or related results in the references  \cite{BarbosaBM-2008}, \cite{BarbosaKO-1991}, \cite{BarbosaGS-1987}, \cite{doCarmoP:1979}, \cite{doCarmo:1988}, \cite{ChodoshL-2021}, \cite{ChodoshLMD-2024}, \cite{ColbrieS:1979}, \cite{CoswasckF-2021}, \cite{IliasNS:2022}, \cite{MeeksPR:2008}, \cite{MeeksPR:2016}, \cite{Oshikiri-1982}, \cite{Oshikiri:2002}, \cite{Pogorelov:1981}, \cite{SampaioS-2024}, \cite{SampaioS-2025}, \cite{SampaioS-2026}, \cite{Ye:1996}, etc.

In this article, we give a complete positive answer to Conjecture \ref{conjecture:yau-docarmo-solomon} and prove that Conjecture \ref{conjecture:cmc_solomon} has a positive answer whether $n\leq 6$. More precisely, we prove the following:

\begin{customthm}{\ref*{thm:main_euclidean}}
Let $\mathfrak{F}$ be a codimension one CMC foliation of $\mathbb{R}^{n+1}$. Then, each leaf of $\mathfrak{F}$ is a minimal and stable hypersurface. 
\end{customthm}
\begin{customcor}{\ref*{thm:main_euclidean_low_dim}}
Let $\mathfrak{F}$ be a codimension one CMC foliation of $\mathbb{R}^{n+1}$. If $n\leq 6$, then $\mathfrak{F}$ consists entirely of parallel hyperplanes.
\end{customcor}

\section{Preliminaries}
\subsection{Setting and notation}

Throughout, $n\ge1$ and $\mathfrak{F}$ is a smooth, transversely oriented foliation of $\R^{n+1}$ by connected hypersurfaces (the leaves), and $N$ is a smooth unit vector field on $\R^{n+1}$ normal to the leaves. Each leaf is an injectively immersed hypersurface, which we regard as a connected manifold with its leaf topology; none of the arguments below requires the leaves to be properly embedded. We write $\nabla$, $\langle\cdot,\cdot\rangle$ and $\divg$ for the Euclidean connection, metric and divergence. The function $H\colon\R^{n+1}\to\R$ is defined by
\begin{equation}\label{eq:H}
\divg N=-nH ;
\end{equation}
it is smooth. Our standing hypothesis is $\mathfrak{F}$ is a CMC foliation, i.e, $H$ is constant along each leaf of $\mathfrak{F}$.

We extend the shape operator to all of $T\R^{n+1}$ by $A:=-\nabla N$, that is, $AX=-\nabla_XN$. Since $\langle\nabla_XN,N\rangle=\tfrac12X\langle N,N\rangle=0$ for every $X$, the image of $A$ at a point is contained in the tangent space of the leaf through that point. For $X$ tangent to a leaf $L$, $AX$ is the shape operator of $L$ with respect to $N$; thus $A|_{TL}$ is a symmetric endomorphism of $TL$, and
\begin{equation}\label{eqtraco}
\tr A|_{TL}=-\sum_{i=1}^n\langle\nabla_{e_i}N,e_i\rangle=
-\divg N=nH,        
\end{equation}

for any orthonormal frame $\{e_i\}$ of $TL$. Hence $H$ is the mean curvature of the leaves with respect to $N$. Note also that $AN=-\nabla_NN$ is tangent to the leaves.

We fix the following notation:
\begin{align*}
g&:=N(H), & T&:=\nabla_NN \quad(\text{tangent to the leaves}),\\
|A|^2&:=\tr\big((A|_{TL})^2\big), & k_1,\dots,k_n&:=\text{eigenvalues of }A|_{TL},
\end{align*}
and $\gl$, $\divg_L$, $\lapL$ denote the gradient, divergence and Laplacian of the leaf $L$ (with its induced metric), applied to the restrictions to $L$ of functions and of vector fields tangent to the leaves. For a symmetric endomorphism (or bilinear form) $S$ of $T_x\R^{n+1}$ we write $\|S\|$ for its operator norm and $\tr_{TL}S:=\sum_{i=1}^n\langle Se_i,e_i\rangle$ for its trace over the tangent space of the leaf through $x$; note that $|\tr_{TL}S|\le n\|S\|$ and that $\tr_{TL}S=\tr S-\langle SN,N\rangle$, so $\tr_{TL}\Hess f$ is a smooth function on $\R^{n+1}$ for every smooth $f$. We shall use repeatedly the elementary inequality
\begin{equation}\label{eq:CS_long}
|A|^2=\sum_{i=1}^n k_i^2\ \ge\ \frac1n\Big(\sum_{i=1}^n k_i\Big)^2=nH^2 .
\end{equation}

\subsection{Preliminary identities}

\begin{lemma}\label{lem:gradH}
$\nabla H=gN$.
\end{lemma}
\begin{proof}
Since $\mathfrak{F}$ is a CMC foliation, $H$ is constant along each leaf. Hence, \(X(H)=\langle\nabla H,X\rangle=0\) for every vector field \(X\) tangent to the leaves. Therefore, \(\nabla H\) is orthogonal to the leaves and thus parallel to \(N\). Since \(\langle\nabla H,N\rangle=N(H)=g\) and \(|N|=1\), we conclude that \(\nabla H=gN\).
\end{proof}

The following lemma was proved by Oshikiri in \cite{Oshikiri:2002}.

\begin{lemma}[Proposition 2 in \cite{Oshikiri:2002}]\label{lem:gT}
For every vector field $X$ tangent to a leaf $L$ of $\mathfrak F$, we have
\[
X(g)=g\langle T,X\rangle.
\]
Equivalently, the gradient of $g$ along each leaf satisfies
\[
\nabla_L g=gT.
\]
Consequently, $g$ has constant sign on each leaf $L$: either
$g\equiv0$ on $L$, or $g>0$ everywhere on $L$, or $g<0$
everywhere on $L$.
\end{lemma}

The following two lemmas were proved by Barbosa, Kenmotsu, and Oshikiri in \cite{BarbosaKO-1991}.
\begin{lemma}[Proposition 2.14 in \cite{BarbosaKO-1991}]\label{lem:divT}
$\divg T=|A|^2-ng$.
\end{lemma}

\begin{lemma}[Proposition 2.14 in \cite{BarbosaKO-1991}]\label{lem:divX}
Let $X$ be a $C^1$ vector field on $\R^{n+1}$ tangent to the leaves. Then, along every leaf $L$,
\begin{equation}\label{eq:divX}
\divg X=\divg_LX-\langle X,T\rangle .
\end{equation}
\end{lemma}
The following lemma may be well-known, but we could not find a proper reference for it; thus, we present a proof here. 
\begin{lemma}\label{lem:lapL}
For $f\in C^\infty(\mathbb R^{n+1})$ and any leaf $L\in \mathfrak{F}$, we have
\begin{equation}\label{eq:lapL}
\Delta_L f = \operatorname{tr}_{TL}\operatorname{Hess} f
+nH\,N(f) \qquad\text{on }L.
\end{equation}
In particular, for $\rho(x):=\sqrt{1+|x|^2}$, one has $\Delta_L\rho\le n(1+|H|).$
\end{lemma}

\begin{proof}
Let $\{e_1,\ldots,e_n\}$ be a local orthonormal frame of $TL$.
By the definition of the Laplace--Beltrami operator on $L$,
$$
\Delta_L f
=
\sum_{i=1}^n
\left(
e_i(e_i f)-(\nabla^L_{e_i}e_i)(f)
\right),
$$
where $\nabla^L$ denotes the Levi-Civita connection of the induced
metric on $L$. On the other hand, the ambient Hessian satisfies
$$
\operatorname{Hess}f(e_i,e_i)=e_i(e_i f)-(\nabla_{e_i}e_i)(f),
$$
where $\nabla$ is the Euclidean connection of $\mathbb R^{n+1}$. Since $\nabla_{e_i}e_i$ decomposes into its tangential and normal
components, we have
$$
\nabla_{e_i}e_i=\nabla^L_{e_i}e_i+\langle\nabla_{e_i}e_i,N\rangle N.
$$
Differentiating $\langle e_i,N\rangle=0$ along $e_i$, we obtain
$$
\langle\nabla_{e_i}e_i,N\rangle=-\langle e_i,\nabla_{e_i}N\rangle=\langle Ae_i,e_i\rangle,
$$
where $A(X)=-\nabla_XN$ is the shape operator. Consequently,
$$
\nabla_{e_i}e_i=\nabla^L_{e_i}e_i+\langle Ae_i,e_i\rangle N.
$$
It follows that
\begin{equation}
e_i(e_i f)-(\nabla^L_{e_i}e_i)(f)=\operatorname{Hess}f(e_i,e_i)+\langle Ae_i,e_i\rangle N(f).
\end{equation}
Summing over $i$, we get

\begin{eqnarray*}
\Delta_L f & = & \sum_{i=1}^n\operatorname{Hess}f(e_i,e_i)+\left(\sum_{i=1}^n\langle Ae_i,e_i\rangle\right)N(f)\\
& = & \operatorname{tr}_{TL}\operatorname{Hess}f
+nH\,N(f),    
\end{eqnarray*}
where we used the equation $\operatorname{tr}(A|_{TL})=nH$ (see Equation \eqref{eqtraco}). This proves the Equation \eqref{eq:lapL}.

We now compute the gradient and the Hessian of
$$
\rho(x):=\sqrt{1+|x|^2},
\qquad x=(x_1,\ldots,x_{n+1})\in\mathbb R^{n+1}.
$$

By the chain rule, for every $i=1,\ldots,n+1$, we have $\frac{\partial\rho}{\partial x_i} =\frac{x_i}{\rho}$, and thus
$$
\nabla\rho=\frac{x}{\rho}.
$$
Moreover,
\[
\frac{\partial^2\rho}{\partial x_i\partial x_j}
=
\frac{\delta_{ij}}{\rho}
-\frac{x_ix_j}{\rho^3}.
\]
Hence, the Hessian matrix is
\[
\operatorname{Hess}\rho
=
\frac1\rho I-\frac1{\rho^3}x\otimes x,
\]
where $I$ denotes the identity operator and $x\otimes x$ is the linear
operator defined by
\[
(x\otimes x)(v)=\langle x,v\rangle x,
\qquad v\in\mathbb R^{n+1}.
\]
Equivalently, for any $v,w\in\mathbb R^{n+1}$,
\[
\operatorname{Hess}\rho(v,w)
=
\frac{\langle v,w\rangle}{\rho}
-
\frac{\langle x,v\rangle\langle x,w\rangle}{\rho^3}.
\]
In particular,
\[
\operatorname{Hess}\rho(v,v)
=
\frac{|v|^2}{\rho}
-
\frac{\langle x,v\rangle^2}{\rho^3}.
\]
Therefore,
\[
\begin{aligned}
\operatorname{Hess}\rho(v,v)
&\ge
\left(\frac{1}{\rho}-\frac{|x|^2}{\rho^3}\right)|v|^2\\
&=\frac{\rho^2-|x|^2}{\rho^3}|v|^2\\
&=\frac{1}{\rho^3}|v|^2\ge0.
\end{aligned}
\]
Moreover, since $\rho\ge1$, we have
\[
\operatorname{Hess}\rho(v,v)
\le \frac{|v|^2}{\rho}
\le |v|^2.
\]
Hence, in the sense of quadratic forms,
\[
0\le\operatorname{Hess}\rho\le I.
\]
Taking the trace over an orthonormal basis
$\{e_1,\ldots,e_n\}$ of $TL$, we obtain
\[
\operatorname{tr}_{TL}\operatorname{Hess}\rho
=
\sum_{i=1}^n\operatorname{Hess}\rho(e_i,e_i)
\le
\sum_{i=1}^n|e_i|^2
=n.
\]
Furthermore,
\[
|N(\rho)|
=
|\langle\nabla\rho,N\rangle|
\le|\nabla\rho|
=
\frac{|x|}{\rho}
\le1.
\]
Applying \eqref{eq:lapL} to $\rho$, we conclude that
\[
\begin{aligned}
\Delta_L\rho
&=
\operatorname{tr}_{TL}\operatorname{Hess}\rho
+nH\,N(\rho)\\
&\le n+n|H|\,|N(\rho)|\\
&\le n(1+|H|).
\end{aligned}
\]
This completes the proof.
\end{proof}

The next result is straightforward, but for completeness, we present a proof here. 

\begin{lemma}\label{lem:hessian}
Let $\chi\in C^{\infty}(\mathbb{R}^{n+1})$ and let $q\geq 4$ be an
integer. Then
\[
\operatorname{Hess}(\chi^q)
=
q(q-1)\chi^{q-2}\,d\chi\otimes d\chi
+
q\chi^{q-1}\operatorname{Hess}\chi.
\]
Here, the Hessian is regarded as a symmetric bilinear form, and
$d\chi\otimes d\chi$ denotes the bilinear form defined by
\[
(d\chi\otimes d\chi)(v,w)=d\chi(v)d\chi(w).
\]
\end{lemma}

\begin{proof}
Let $v,w\in\mathbb{R}^{n+1}$ be arbitrary vectors. By the
definition of the Hessian,
\[
\operatorname{Hess}(\chi^q)(v,w)
=
\left\langle
D_v\bigl(\nabla(\chi^q)\bigr),w
\right\rangle.
\]
By the chain rule,
\[
\nabla(\chi^q)=q\chi^{q-1}\nabla\chi.
\]
Differentiating this identity in the direction of $v$ and applying
the product rule, we obtain
\[
D_v\bigl(\nabla(\chi^q)\bigr)
=
q(q-1)\chi^{q-2}d\chi(v)\nabla\chi
+
q\chi^{q-1}D_v(\nabla\chi).
\]
Taking the inner product with $w$ gives
\begin{align*}
\operatorname{Hess}(\chi^q)(v,w)
={}&q(q-1)\chi^{q-2}
d\chi(v)\langle\nabla\chi,w\rangle
+q\chi^{q-1}
\langle D_v(\nabla\chi),w\rangle.
\end{align*}
Since
\[
d\chi(w)=\langle\nabla\chi,w\rangle
\]
and
\[
\operatorname{Hess}\chi(v,w)
=
\langle D_v(\nabla\chi),w\rangle,
\]
we conclude that
\[
\begin{aligned}
\operatorname{Hess}(\chi^q)(v,w)
={}&q(q-1)\chi^{q-2}d\chi(v)d\chi(w)
+q\chi^{q-1}\operatorname{Hess}\chi(v,w).
\end{aligned}
\]
By the definition of $d\chi\otimes d\chi$, this is equivalent to
\[
\operatorname{Hess}(\chi^q)
=
q(q-1)\chi^{q-2}\,d\chi\otimes d\chi
+
q\chi^{q-1}\operatorname{Hess}\chi.
\]
Since $v$ and $w$ were arbitrary, the identity holds everywhere
on $\mathbb{R}^{n+1}$.
\end{proof}

\section{Key results}\label{sec:flux}
In this Section, we now introduce the main tools and results needed to establish our main theorem. 

Throughout this section and the next, we fix an odd integer \(p>n\) (for instance, \(p=n+1\) if \(n\) is even and \(p=n+2\) if \(n\) is odd), and set \(q:=p+1\). Thus, \(p\geq 3\), and both \(p-1\) and \(q\) are even.

\subsection{A vector field with nonnegative divergence}

\begin{lemma}\label{lem:divV}
The vector field
\begin{equation}\label{eq:V}
V:=H^pN+\frac{p}{n}H^{p-1}T
\end{equation}
is smooth on $\mathbb R^{n+1}$, and
\begin{equation}\label{eq:divV}
\operatorname{div}V
=
\frac{p}{n}H^{p-1}|A|^2-nH^{p+1}
\ge (p-n)H^{p+1}
=(p-n)|H|^{p+1}
\ge0.
\end{equation}
\end{lemma}

\begin{proof}
Since $H$, $N$, and $T$ are smooth and $p$ is an integer, the
vector field $V$ defined in \eqref{eq:V} is smooth on
$\mathbb R^{n+1}$. We first compute the divergence of each term in $V$. By the
product rule for the divergence,
\[
\operatorname{div}(H^pN)
=
\langle\nabla(H^p),N\rangle
+H^p\operatorname{div}N.
\]
Using the chain rule and the identities 
$\nabla H=gN$ (see Lemma \ref{lem:gradH}) and $\operatorname{div}N=-nH$, we obtain
\[
\begin{aligned}
\operatorname{div}(H^pN)
&=pH^{p-1}\langle\nabla H,N\rangle
-nH^{p+1}\\
&=pH^{p-1}g\langle N,N\rangle
-nH^{p+1}\\
&=pH^{p-1}g-nH^{p+1},
\end{aligned}
\]
where we used $|N|=1$.

Next, applying the product rule to the second term gives
\[
\operatorname{div}(H^{p-1}T)
=
\langle\nabla(H^{p-1}),T\rangle
+H^{p-1}\operatorname{div}T.
\]
Since
\[
\nabla(H^{p-1})
=(p-1)H^{p-2}\nabla H
=(p-1)H^{p-2}gN
\]
and $T\perp N$, we have
\[
\langle\nabla(H^{p-1}),T\rangle
=(p-1)H^{p-2}g\langle N,T\rangle=0.
\]
Therefore, by Lemma~\ref{lem:divT},
\[
\begin{aligned}
\operatorname{div}\left(\frac{p}{n}H^{p-1}T\right)
&=\frac{p}{n}H^{p-1}\operatorname{div}T\\
&=\frac{p}{n}H^{p-1}(|A|^2-ng)\\
&=\frac{p}{n}H^{p-1}|A|^2
-pH^{p-1}g.
\end{aligned}
\]

Adding the two divergence identities, we find that
$$
\operatorname{div}V=pH^{p-1}g-nH^{p+1}+\frac{p}{n}H^{p-1}|A|^2-pH^{p-1}g=\frac{p}{n}H^{p-1}|A|^2-nH^{p+1}.
$$
Thus, the terms involving $g$ cancel identically, and the
equality in \eqref{eq:divV} follows without any assumption on
the signs of $g$ or $H$.

It remains to establish the inequality. By the
Cauchy--Schwarz inequality,
\[
\bigl(\operatorname{tr}(A|_{TL})\bigr)^2
\le n\,|A|^2.
\]
Since $\operatorname{tr}(A|_{TL})=nH$, this yields
\begin{equation}\label{eq:CS}
|A|^2\ge nH^2.
\end{equation}
Because $p$ is odd, $p-1$ is even, and hence
$H^{p-1}\ge0$. Multiplying \eqref{eq:CS} by
$\frac{p}{n}H^{p-1}\ge0$, we obtain
\[
\frac{p}{n}H^{p-1}|A|^2
\ge pH^{p+1}.
\]
Consequently,
\[
\begin{aligned}
\operatorname{div}V
&=\frac{p}{n}H^{p-1}|A|^2-nH^{p+1}\\
&\ge pH^{p+1}-nH^{p+1}\\
&=(p-n)H^{p+1}.
\end{aligned}
\]
Finally, since $p+1$ is even,
\[
H^{p+1}=|H|^{p+1}.
\]
Moreover, $p>n$ implies $p-n>0$. Therefore,
\[
\operatorname{div}V
\ge(p-n)|H|^{p+1}\ge0,
\]
which completes the proof.
\end{proof}

\subsection{Removing \texorpdfstring{$T$}{T} from boundary terms}

The field $T$ is not a priori bounded, so terms of the form $\langle\nabla\chi,T\rangle$ must be handled with care. The following identity trades $T$ for Hessians and normal derivatives of a test function.

\begin{lemma}\label{lem:T}
Let $\chi\in C^\infty_c(\mathbb R^{n+1})$ and let
$\kappa\in C^\infty(\mathbb R^{n+1})$ be constant along the
leaves of the foliation. Then
\begin{equation}\label{eq:T}
\begin{aligned}
\kappa\langle\nabla\chi,T\rangle
={}&\kappa\,\operatorname{tr}_{TL}\operatorname{Hess}\chi
+nH\kappa\,N(\chi)
-\operatorname{div}(\kappa\nabla_L\chi),
\end{aligned}
\end{equation}
where
\[
\nabla_L\chi:=\nabla\chi-N(\chi)N
\]
is the leafwise gradient of $\chi$. In particular,
$\kappa\nabla_L\chi$ is a smooth compactly supported vector
field on $\mathbb R^{n+1}$ tangent to the leaves, and
\[
\int_{\mathbb R^{n+1}}
\operatorname{div}(\kappa\nabla_L\chi)\,d x=0.
\]
\end{lemma}

\begin{proof}
We first observe that
\[
\nabla_L\chi=\nabla\chi-N(\chi)N
\]
is the orthogonal projection of $\nabla\chi$ onto the tangent
space of each leaf. In particular,
$
\langle\nabla_L\chi,N\rangle=0.
$
Since $T$ is tangent to the leaves, we also have
\[
\langle\nabla\chi,T\rangle
=
\langle\nabla_L\chi,T\rangle.
\]
On each leaf $L$, the restriction of $\nabla_L\chi$ is the
intrinsic gradient of $\chi|_L$. Therefore,
\[
\operatorname{div}_L(\nabla_L\chi)=\Delta_L\chi.
\]
Applying Lemma~\ref{lem:divX} to the tangential vector field
$X=\nabla_L\chi$, we obtain
\[
\operatorname{div}(\nabla_L\chi)
=
\operatorname{div}_L(\nabla_L\chi)
-\langle\nabla_L\chi,T\rangle.
\]
Equivalently,
\[
\langle\nabla_L\chi,T\rangle
=
\Delta_L\chi-\operatorname{div}(\nabla_L\chi).
\]
Consequently,
\[
\langle\nabla\chi,T\rangle
=
\Delta_L\chi-\operatorname{div}(\nabla_L\chi).
\]

By Lemma~\ref{lem:lapL}, the intrinsic Laplacian satisfies
\[
\Delta_L\chi
=
\operatorname{tr}_{TL}\operatorname{Hess}\chi
+nH\,N(\chi).
\]
Substituting this identity, we conclude that
\[
\begin{aligned}
\langle\nabla\chi,T\rangle
={}&\operatorname{tr}_{TL}\operatorname{Hess}\chi
+nH\,N(\chi)
-\operatorname{div}(\nabla_L\chi).
\end{aligned}
\]

Multiplying by $\kappa$, we obtain
\[
\begin{aligned}
\kappa\langle\nabla\chi,T\rangle
={}&\kappa\,\operatorname{tr}_{TL}\operatorname{Hess}\chi
+nH\kappa\,N(\chi)-\kappa\operatorname{div}(\nabla_L\chi).
\end{aligned}
\]

It remains to rewrite the last term. By the product rule
for the divergence,
\[
\operatorname{div}(\kappa\nabla_L\chi)
=
\langle\nabla\kappa,\nabla_L\chi\rangle
+\kappa\operatorname{div}(\nabla_L\chi).
\]
Since $\kappa$ is constant along the leaves, its derivative
in every tangential direction vanishes. Thus,
\[
\langle\nabla\kappa,X\rangle=0
\qquad\text{for every }X\in TL.
\]
Equivalently, $\nabla\kappa$ is parallel to $N$. Because
$\nabla_L\chi$ is tangent to the leaves, it follows that
\[
\langle\nabla\kappa,\nabla_L\chi\rangle=0.
\]
Hence,
\[
\operatorname{div}(\kappa\nabla_L\chi)
=
\kappa\operatorname{div}(\nabla_L\chi).
\]
Substituting this equality into the preceding expression
gives
\[
\begin{aligned}
\kappa\langle\nabla\chi,T\rangle
={}&\kappa\,\operatorname{tr}_{TL}\operatorname{Hess}\chi
+nH\kappa\,N(\chi)-\operatorname{div}(\kappa\nabla_L\chi),
\end{aligned}
\]
which proves \eqref{eq:T}.

Finally, since $\chi$ is compactly supported and
$\nabla_L\chi$ is obtained from $\nabla\chi$ by orthogonal
projection, $\nabla_L\chi$ is also compactly supported.
Multiplication by the smooth function $\kappa$ preserves
compact support. Thus $\kappa\nabla_L\chi$ is a smooth
compactly supported vector field on $\mathbb R^{n+1}$.
The divergence theorem, applied on a sufficiently large
ball containing its support, yields
\[
\int_{\mathbb R^{n+1}}
\operatorname{div}(\kappa\nabla_L\chi)\,dx=0.
\]
This completes the proof.
\end{proof}

\subsection{The integral inequality}

\begin{proposition}\label{prop:int}
For every $\chi\in C^\infty_c(\mathbb R^{n+1})$, we have
\begin{equation}\label{eq:int}
\begin{aligned}
(p-n)\int_{\mathbb R^{n+1}}\chi^q H^{p+1}
\le{}&(p+1)\int_{\mathbb R^{n+1}}
|H|^p\,|\nabla(\chi^q)|+p\int_{\mathbb R^{n+1}}
H^{p-1}\,\|\operatorname{Hess}\chi^q\|.
\end{aligned}
\end{equation}
\end{proposition}

\begin{proof}
Let $\chi\in C^\infty_c(\mathbb R^{n+1})$ and recall that
$q=p+1$. Consider the vector field
\[
W:=\chi^qV,
\]
where
$
V=H^pN+\frac{p}{n}H^{p-1}T
$
is the smooth vector field defined in Lemma~\ref{lem:divV}.
Since $\chi$ has compact support, $W$ is a smooth compactly
supported vector field on $\mathbb R^{n+1}$. Hence, by the
divergence theorem,
\begin{equation}\label{eq:integraldivW}
\int_{\mathbb R^{n+1}}\operatorname{div}W=0.
\end{equation}

By the product rule for the divergence,
\[
\operatorname{div}W
=
\chi^q\operatorname{div}V
+\langle\nabla(\chi^q),V\rangle.
\]
Using the definition of $V$, we obtain
\[
\begin{aligned}
\operatorname{div}W
={}&\chi^q\operatorname{div}V
+H^pN(\chi^q)+\frac{p}{n}H^{p-1}
\langle\nabla(\chi^q),T\rangle.
\end{aligned}
\]

We now apply Lemma~\ref{lem:T} with the test function $\chi^q$
and the function
\[
\kappa:=\frac{p}{n}H^{p-1}.
\]
The function $\kappa$ is smooth and constant along the leaves,
because $H$ is constant on each leaf. Thus, Lemma~\ref{lem:T}
gives

\begin{equation*}
\frac{p}{n}H^{p-1}
\langle\nabla(\chi^q),T\rangle  =  \frac{p}{n}H^{p-1}
\operatorname{tr}_{TL}\operatorname{Hess}(\chi^q)
+pH^pN(\chi^q)-
\operatorname{div}\left(
\frac{p}{n}H^{p-1}\nabla_L(\chi^q)
\right).    
\end{equation*}

Substituting this identity into the expression for
$\operatorname{div}W$, we find

\begin{equation*}
\operatorname{div}W=\chi^q\operatorname{div}V
+(p+1)H^pN(\chi^q)+\frac{p}{n}H^{p-1}
\operatorname{tr}_{TL}\operatorname{Hess}(\chi^q)-\operatorname{div}\left(
\frac{p}{n}H^{p-1}\nabla_L(\chi^q)
\right).    
\end{equation*}

Integrating over $\mathbb R^{n+1}$, the divergence theorem
and the compact support of $\chi$ imply
\[
\int_{\mathbb R^{n+1}}
\operatorname{div}\left(
\frac{p}{n}H^{p-1}\nabla_L(\chi^q)
\right)=0.
\]
Together with \eqref{eq:integraldivW}, this yields
\begin{equation}\label{eq:intid}
\begin{aligned}
\int_{\mathbb R^{n+1}}\chi^q\operatorname{div}V
={}&-(p+1)\int_{\mathbb R^{n+1}}H^pN(\chi^q)-\frac{p}{n}\int_{\mathbb R^{n+1}}
H^{p-1}\operatorname{tr}_{TL}
\operatorname{Hess}(\chi^q).
\end{aligned}
\end{equation}

We next estimate both sides of \eqref{eq:intid}. Since $q$ is
even, $\chi^q\ge0$. Therefore, by inequality \eqref{eq:divV} in Lemma~\ref{lem:divV} implies
\[
\begin{aligned}
(p-n)\int_{\mathbb R^{n+1}}\chi^qH^{p+1}
&\le
\int_{\mathbb R^{n+1}}\chi^q\operatorname{div}V.
\end{aligned}
\]
Combining this inequality with \eqref{eq:intid}, we obtain
\[
\begin{aligned}
&(p-n)\int_{\mathbb R^{n+1}}\chi^qH^{p+1} \le
-(p+1)\int_{\mathbb R^{n+1}}H^pN(\chi^q)-\frac{p}{n}\int_{\mathbb R^{n+1}}
H^{p-1}\operatorname{tr}_{TL}
\operatorname{Hess}(\chi^q).
\end{aligned}
\]

To estimate the first integral on the right-hand side, we use
the fact that $N$ is a unit vector field. Thus,
\[
|N(\chi^q)|
=
|\langle\nabla(\chi^q),N\rangle|
\le|\nabla(\chi^q)|.
\]
Consequently,
\[
-\int_{\mathbb R^{n+1}}H^pN(\chi^q)
\le
\int_{\mathbb R^{n+1}}
|H|^p\,|\nabla(\chi^q)|.
\]

For the second integral, recall that $p-1$ is even, so
$H^{p-1}\ge0$. Moreover, for any symmetric endomorphism $S$
of an $n$-dimensional inner product space,
\[
|\operatorname{tr}S|\le n\|S\|,
\]
where $\|S\|$ denotes the operator norm. Applying this
inequality to the restriction of
$\operatorname{Hess}(\chi^q)$ to $TL$, we obtain
\[
\left|
\operatorname{tr}_{TL}\operatorname{Hess}(\chi^q)
\right|
\le
n\|\operatorname{Hess}(\chi^q)\|.
\]
It follows that
\[
\begin{aligned}
&-\frac{p}{n}\int_{\mathbb R^{n+1}}
H^{p-1}\operatorname{tr}_{TL}
\operatorname{Hess}(\chi^q)\le
p\int_{\mathbb R^{n+1}}
H^{p-1}\|\operatorname{Hess}(\chi^q)\|.
\end{aligned}
\]

Combining the preceding estimates, we conclude that
\[
\begin{aligned}
&(p-n)\int_{\mathbb R^{n+1}}\chi^qH^{p+1}\le
(p+1)\int_{\mathbb R^{n+1}}
|H|^p\,|\nabla(\chi^q)|+
p\int_{\mathbb R^{n+1}}
H^{p-1}\|\operatorname{Hess}(\chi^q)\|.
\end{aligned}
\]
This proves \eqref{eq:int}.
\end{proof}

\subsection{Estimates for the gradient and Hessian}\label{subsec:estimates}

Fix $\chi_0\in C^\infty([0,\infty))$ such that
\[
0\le\chi_0\le1,\qquad
\chi_0=1\text{ on }[0,1],\qquad
\chi_0=0\text{ on }[2,\infty),
\]
and set
\[
C_0:=
\sup_{s\ge0}|\chi_0'(s)|
+
\sup_{s\ge0}|\chi_0''(s)|.
\]
For $R>0$, define
\[
\chi_R(x):=\chi_0\left(\frac{|x|}{R}\right),
\qquad x\in\mathbb R^{n+1}.
\]
Then $\chi_R\in C^\infty_c(\mathbb R^{n+1})$. Indeed,
$\chi_R$ is constant in a neighborhood of the origin and
vanishes whenever $|x|\ge2R$. Moreover,
\[
0\le\chi_R\le1,\qquad
\chi_R=1\text{ on }B_R,\qquad
\operatorname{supp}\chi_R\subset\overline{B_{2R}},
\]
where $B_R$ denotes the open ball of radius $R$ centered at
the origin.

\begin{lemma}Under the above notation, we have
\begin{equation}\label{eq:cutoff}
|\nabla\chi_R|\le\frac{C_0}{R},
\qquad
\|\operatorname{Hess}\chi_R\|\le\frac{C_0}{R^2}.
\end{equation}
\end{lemma}
\begin{proof}
    To prove these estimates, let $x\ne0$ and write
\[
r:=|x|,
\qquad
\widehat{x}:=\frac{x}{r}.
\]
Thus $|\widehat{x}|=1$. Since $r=\left(\sum_{i=1}^{n+1}x_i^2\right)^{1/2},$ we have
\[
\frac{\partial r}{\partial x_i}
=
\frac{x_i}{r},
\qquad i=1,\ldots,n+1.
\]
Consequently,
\[
\nabla r=\frac{x}{r}=\widehat{x}.
\]

By the chain rule,
\[
\begin{aligned}
\nabla\chi_R
&=
\chi_0'\left(\frac{r}{R}\right)
\nabla\left(\frac{r}{R}\right)\\
&=
\frac{1}{R}\chi_0'\left(\frac{r}{R}\right)\nabla r\\
&=
\frac{\chi_0'(r/R)}{R}\widehat{x}.
\end{aligned}
\]
Since $|\widehat{x}|=1$, it follows that
\[
|\nabla\chi_R|
=
\frac{1}{R}\left|\chi_0'\left(\frac{r}{R}\right)\right|
\le
\frac{\sup_{s\ge0}|\chi_0'(s)|}{R}
\le\frac{C_0}{R}.
\]
This proves the first estimate in \eqref{eq:cutoff}.

We now compute the Hessian. Recall that, for vectors
$u,v\in\mathbb R^{n+1}$, the tensor product $u\otimes v$
defines the linear operator
\[
(u\otimes v)(w):=\langle v,w\rangle u.
\]
In particular,
\[
(\widehat{x}\otimes\widehat{x})(v)
=
\langle\widehat{x},v\rangle\widehat{x}.
\]
Since $|\widehat{x}|=1$, the operator
$\widehat{x}\otimes\widehat{x}$ is the orthogonal projection
onto the radial direction $\operatorname{span}\{\widehat{x}\}$,
whereas
\[
I-\widehat{x}\otimes\widehat{x}
\]
is the orthogonal projection onto the orthogonal complement
$\widehat{x}^{\perp}$. For convenience, define
\[
a(r):=\frac{1}{R}\chi_0'\left(\frac{r}{R}\right).
\]
Then
\[
\nabla\chi_R=a(r)\widehat{x}.
\]
For any $v\in\mathbb R^{n+1}$, the product rule gives
\begin{eqnarray*}
\operatorname{Hess}\chi_R(v)=D(\nabla\chi_R)(v) & = & Da(v)\,\widehat{x}+a(r)D\widehat{x}(v)\\
& = & a'(r)Dr(v)\,\widehat{x}+a(r)D\widehat{x}(v).
\end{eqnarray*}
First, by the chain rule,
\[
a'(r)=\frac{1}{R^2}\chi_0''\left(\frac{r}{R}\right).
\]
Since $Dr(v)=\langle \nabla r,v\rangle =\langle\widehat{x},v\rangle$, we obtain
\[
Da(v)
=
\frac{1}{R^2}\chi_0''\left(\frac{r}{R}\right)
\langle\widehat{x},v\rangle.
\]
Therefore,

\begin{equation}\label{eq.diva}
Da(v)\,\widehat{x}  =  
\frac{1}{R^2}\chi_0''\left(\frac{r}{R}\right)
\langle\widehat{x},v\rangle\widehat{x}\\
= \frac{1}{R^2}\chi_0''\left(\frac{r}{R}\right)
(\widehat{x}\otimes\widehat{x})(v).    
\end{equation}

Next, since $\widehat{x}=x/r$, differentiation in the
direction $v$ yields
\[
D\widehat{x}(v)
=
\frac{v}{r}-\frac{x}{r^2}Dr(v).
\]
Using $x=r\widehat{x}$ and
$Dr(v)=\langle\widehat{x},v\rangle$, we find that
\[
\begin{aligned}
D\widehat{x}(v)
&=
\frac{v}{r}
-\frac{\widehat{x}}{r}
\langle\widehat{x},v\rangle\\
&=
\frac{1}{r}
\left(v-\langle\widehat{x},v\rangle\widehat{x}\right)\\
&=
\frac{1}{r}
\left(I-\widehat{x}\otimes\widehat{x}\right)(v).
\end{aligned}
\]
Hence,
\[
D\widehat{x}
=
\frac{1}{r}
\left(I-\widehat{x}\otimes\widehat{x}\right).
\]
Multiplying by $a(r)$, we obtain
\begin{equation}\label{eq.divxchapeu}
a(r)D\widehat{x}
=
\frac{\chi_0'(r/R)}{Rr}
\left(I-\widehat{x}\otimes\widehat{x}\right).    
\end{equation}

Combining the two Equations \eqref{eq.diva} and \eqref{eq.divxchapeu} in the derivative of
$\nabla\chi_R$, we conclude that
\begin{equation}\label{eq:hesscutoff}
\begin{aligned}
\operatorname{Hess}\chi_R
={}&
\frac{\chi_0''(r/R)}{R^2}
\widehat{x}\otimes\widehat{x}+
\frac{\chi_0'(r/R)}{Rr}
\left(I-\widehat{x}\otimes\widehat{x}\right).
\end{aligned}
\end{equation}

We next estimate the operator norm of this Hessian.
Set
\[
P:=\widehat{x}\otimes\widehat{x},
\qquad
Q:=I-P,
\]
and
\[
\alpha:=\frac{\chi_0''(r/R)}{R^2},
\qquad
\beta:=\frac{\chi_0'(r/R)}{Rr}.
\]
Then \eqref{eq:hesscutoff} becomes
\[
\operatorname{Hess}\chi_R=\alpha P+\beta Q.
\]
The operators $P$ and $Q$ are orthogonal projections onto
complementary orthogonal subspaces. In particular, for every
$v\in\mathbb R^{n+1}$,
\[
v=Pv+Qv,\qquad
\langle Pv,Qv\rangle=0,
\]
and therefore
\[
|v|^2=|Pv|^2+|Qv|^2.
\]
It follows that
\[
\begin{aligned}
|\operatorname{Hess}\chi_R(v)|^2
&=|\alpha Pv+\beta Qv|^2\\
&=\alpha^2|Pv|^2+\beta^2|Qv|^2\\
&\le
\max\{|\alpha|^2,|\beta|^2\}
\left(|Pv|^2+|Qv|^2\right)\\
&=
\max\{|\alpha|^2,|\beta|^2\}|v|^2.
\end{aligned}
\]
Taking the supremum over all unit vectors $v$, we obtain
\[
\|\operatorname{Hess}\chi_R\|
\le\max\{|\alpha|,|\beta|\}.
\]
Equality holds because the radial subspace and its orthogonal
complement are nontrivial eigenspaces of the Hessian. Hence,
\[
\|\operatorname{Hess}\chi_R\|
=
\max\left\{
\frac{|\chi_0''(r/R)|}{R^2},
\frac{|\chi_0'(r/R)|}{Rr}
\right\}.
\]

Finally, since $\chi_0$ is constant on $[0,1]$ and on
$[2,\infty)$, both $\nabla\chi_R$ and
$\operatorname{Hess}\chi_R$ vanish outside the annulus $R\le r\le2R.$ On this annulus, $r\ge R$, so
$\frac{1}{Rr}\le\frac{1}{R^2}.$ Consequently,
\[
\begin{aligned}
\|\operatorname{Hess}\chi_R\|
&\le
\frac{1}{R^2}
\max\left\{
\sup_{s\ge0}|\chi_0''(s)|,
\sup_{s\ge0}|\chi_0'(s)|
\right\}\\
&\le\frac{C_0}{R^2}.
\end{aligned}
\]
This proves the second estimate in \eqref{eq:cutoff}, and
completes the proof of both claimed bounds.

\end{proof}

\section{Proof of the conjectures}

\subsection{Proof of Conjecture \ref{conjecture:yau-docarmo-solomon}}

\begin{theorem}\label{thm:main_euclidean}
Let $\mathfrak{F}$ be a codimension one CMC foliation of $\mathbb{R}^{n+1}$. Then, each leaf of $\mathfrak{F}$ is a minimal and stable hypersurface. 
\end{theorem}

\begin{proof}
Fix $R>0$, let $\chi:=\chi_R$, where $\chi_R$ is defined in Subsection \ref{subsec:estimates}, and define
\[
a:=|H|\chi\geq 0.
\]
Recall that $q=p+1$, where $p>n$ is an odd integer. By the estimates in
\eqref{eq:cutoff}, we have
\[
|\nabla\chi|\leq \frac{C_0}{R},
\qquad
\|\operatorname{Hess}\chi\|\leq \frac{C_0}{R^2}.
\]
Since $0\leq\chi\leq 1$, the chain rule yields
\[
\nabla(\chi^q)=q\chi^{q-1}\nabla\chi
=q\chi^p\nabla\chi,
\]
and hence
\[
|\nabla(\chi^q)|
=q\chi^p|\nabla\chi|
\leq \frac{qC_0}{R}\chi^p.
\]

By Lemma \ref{lem:hessian}, we have
\[
\operatorname{Hess}(\chi^q)
=
q(q-1)\chi^{q-2}\,d\chi\otimes d\chi
+
q\chi^{q-1}\operatorname{Hess}\chi.
\]

Taking operator norms and using $q-1=p$ and $q-2=p-1$, we obtain
\begin{align*}
\|\operatorname{Hess}(\chi^q)\|
&\leq q(q-1)\chi^{p-1}|\nabla\chi|^2
+q\chi^p\|\operatorname{Hess}\chi\|\\
&\leq \frac{q(q-1)C_0^2}{R^2}\chi^{p-1}
+\frac{qC_0}{R^2}\chi^p\\
&\leq \frac{C_1}{R^2}\chi^{p-1},
\end{align*}
where $C_1:=q(q-1)C_0^2+qC_0.$ In the last inequality, we used $\chi^p\leq\chi^{p-1}$. Since $p$ is odd, both $p+1$ and $p-1$ are even. Therefore,
\[
|H|^p\chi^p=a^p,
\qquad
H^{p-1}\chi^{p-1}=a^{p-1},
\qquad
\chi^{p+1}H^{p+1}=a^{p+1}.
\]
Moreover, $\chi$ vanishes outside $\overline{B_{2R}}$, and thus
$a$ also vanishes there. Set
\[
I=I(R):=\int_{\mathbb{R}^{n+1}}\chi^{p+1}H^{p+1}
=\int_{\mathbb{R}^{n+1}}a^{p+1}.
\]
In particular, $I<\infty$. Define the constants
\[
A_1:=(p+1)qC_0,
\qquad
A_2:=pC_1.
\]
Applying Proposition~\ref{prop:int} and the preceding estimates, we obtain
\begin{align*}
(p-n)I
&\leq (p+1)\int_{\mathbb{R}^{n+1}}
|H|^p|\nabla(\chi^q)|
+p\int_{\mathbb{R}^{n+1}}
H^{p-1}\|\operatorname{Hess}(\chi^q)\|\\
&\leq \frac{A_1}{R}\int_{B_{2R}}a^p
+\frac{A_2}{R^2}\int_{B_{2R}}a^{p-1}.
\end{align*}

Let
\[
\mathcal V:=|B_{2R}|=\omega_{n+1}(2R)^{n+1},
\]
where $\omega_{n+1}$ denotes the volume of the unit ball in
$\mathbb{R}^{n+1}$. Since $a\geq0$ and
\[
\int_{B_{2R}}a^{p+1}
\leq \int_{\mathbb{R}^{n+1}}a^{p+1}=I,
\]
H\"older's inequality gives
\begin{align*}
\int_{B_{2R}}a^p
&\leq
\left(\int_{B_{2R}}a^{p+1}\right)^{\frac{p}{p+1}}
|B_{2R}|^{\frac{1}{p+1}}\\
&\leq I^{\frac{p}{p+1}}\mathcal V^{\frac{1}{p+1}},
\end{align*}
and, in the same way,
\begin{align*}
\int_{B_{2R}}a^{p-1}
&\leq
\left(\int_{B_{2R}}a^{p+1}\right)^{\frac{p-1}{p+1}}
|B_{2R}|^{\frac{2}{p+1}}\\
&\leq I^{\frac{p-1}{p+1}}\mathcal V^{\frac{2}{p+1}}.
\end{align*}
Now put
\[
X:=I^{\frac{1}{p+1}},
\qquad
Y:=\frac{\mathcal V^{\frac{1}{p+1}}}{R}.
\]
Both $X$ and $Y$ are nonnegative, and $Y>0$. Substituting the
H\"older estimates into the preceding integral inequality yields
\[
(p-n)X^{p+1}
\leq A_1X^pY+A_2X^{p-1}Y^2.
\]

If $X=0$, then $X\leq\Lambda Y$ holds trivially for every
$\Lambda>0$. Suppose, therefore, that $X>0$. Dividing the last
inequality by $X^{p-1}Y^2>0$, we obtain
\[
(p-n)\left(\frac{X}{Y}\right)^2
-A_1\frac{X}{Y}-A_2\leq0.
\]
Writing $s:=X/Y$, this becomes
\[
(p-n)s^2-A_1s-A_2\leq0.
\]
Since $p-n>0$ and $s\geq0$, the quadratic inequality implies
\[
s\leq
\frac{A_1+\sqrt{A_1^2+4(p-n)A_2}}{2(p-n)}
=:\Lambda.
\]
Consequently, in both cases $X=0$ and $X>0$, we have
\[
X\leq\Lambda Y.
\]
Raising this inequality to the power $p+1$ and recalling the
definitions of $X$ and $Y$, we conclude that
\[
I\leq\Lambda^{p+1}\frac{\mathcal V}{R^{p+1}}.
\]
Since $\mathcal V=\omega_{n+1}(2R)^{n+1}$, it follows that
\[
I(R)\leq
\Lambda^{p+1}\omega_{n+1}2^{n+1}R^{n-p}.
\]
The constant $\Lambda$ depends only on $n$, $p$, and the chosen
cutoff function $\chi_0$. Moreover, since $\chi=1$ on $B_R$ and
$H^{p+1}=|H|^{p+1}\geq0$, we have
\begin{equation*}
\int_{B_R}{|H|^{p+1}}\leq I(R)\leq
\Lambda^{p+1}\omega_{n+1}2^{n+1}R^{n-p}.
\label{eq:decay}
\end{equation*}

We now show that $H$ vanishes identically. Fix $R_0>0$. For every
$R\geq R_0$, the inclusion $B_{R_0}\subset B_R$ and the
nonnegativity of $|H|^{p+1}$ imply
\[
0\leq\int_{B_{R_0}}|H|^{p+1}
\leq\int_{B_R}|H|^{p+1}
\leq
\Lambda^{p+1}\omega_{n+1}2^{n+1}R^{n-p}.
\]
Because $p>n$, we have $n-p<0$, and therefore
\[
\lim_{R\to\infty}R^{n-p}=0.
\]
Letting $R\to\infty$ in the preceding inequality gives
\[
\int_{B_{R_0}}|H|^{p+1}=0.
\]
The function $|H|^{p+1}$ is continuous and nonnegative. Hence it
must vanish at every point of $B_{R_0}$; otherwise, continuity
would give a neighborhood on which it is strictly positive,
contradicting the vanishing of its integral. Thus $H=0$ on $B_{R_0}$. Since $R_0>0$ was arbitrary, we conclude that $H\equiv0$ on $\mathbb{R}^{n+1}.$ Therefore, every leaf of $\mathfrak F$ has vanishing mean curvature and is consequently a minimal hypersurface. Finally, stability follows from Theorem 1 of \cite{Oshikiri-1982}.
\end{proof}

\subsection{Proof of Conjecture \ref{conjecture:cmc_solomon} for $n\leq 6$}

\begin{corollary}\label{thm:main_euclidean_low_dim}
Let $\mathfrak{F}$ be a codimension one CMC foliation of $\mathbb{R}^{n+1}$. If $n\leq 6$, then $\mathfrak{F}$ consists entirely of parallel hyperplanes.
\end{corollary}
\begin{proof}
By Theorem \ref{thm:main_euclidean}, the leaves are minimal and stable hypersurfaces. 
We know that in an orientable space (e.g. in $\mathbb{R}^{n+1}$) a hypersurface is orientable if and only if it is two-sided.
Since $\mathfrak{F}$ is transversely oriented, each leaf $L\in \mathfrak{F}$ is oriented. Since $n\leq 6$ and each leaf is a complete, minimal, stable, oriented hypersurface, we have, by the main results in \cite{doCarmoP:1979}, \cite{ColbrieS:1979} (for $n=2$), \cite{ChodoshL-2021} (for $n=3$), \cite{ChodoshLMD-2024} (for $n=4$), \cite{Mazet} (for $n=5$), and \cite{HongLW:2026} (for $n=6$), that each leaf $L$ of the foliation is a hyperplane. 
\end{proof}

\bigskip

\noindent {\bf Acknowledgments:} The authors would like to thank Alexandre Fernandes for his interest in this work.

During the preparation of this work, the authors used Claude, Gemini, and ChatGPT to refine arguments, locate references, and draft and revise the manuscript. All AI-generated suggestions were independently verified by the authors, who take full responsibility for the content and accuracy of the final manuscript.

\end{document}